%% file: main.tex
\documentclass[a4paper]{scrartcl}
\setkomafont{disposition}{\normalfont}%
\usepackage{graphicx}%
\usepackage{amsmath,amssymb,amsfonts}%
\usepackage{amsthm}%
\usepackage{mathrsfs}%
\usepackage{xcolor}%
\usepackage{bbm}
\usepackage{hyperref}
\hypersetup{
    colorlinks=true,
    linkcolor=blue,
    urlcolor=blue,
    citecolor=blue
    }
\usepackage{afterpage}
\usepackage[subrefformat=parens,labelformat=parens,caption=false]{subfig}
\usepackage{tikz}
\usetikzlibrary{arrows, arrows.meta, fit, positioning, shapes, backgrounds}
\usepackage{pgfplots}
\pgfplotsset{compat=1.18} 
\usetikzlibrary{decorations.pathreplacing}
\usetikzlibrary{svg.path}
\usepackage{float}
\usepackage{todonotes}
\usepackage{enumerate}
\usepackage{adjustbox}

\usepackage{array}
\newcolumntype{M}[1]{>{$\displaystyle}p{#1}<{$}}

\usepackage{biblatex}
\newtheorem{theorem}{Theorem}[section]
\newtheorem{lemma}[theorem]{Lemma}

\newtheorem{corollary}[theorem]{Corollary}

\theoremstyle{remark}
\newtheorem{definition}[theorem]{Definition}
\newtheorem*{example}{Example}

\newtheorem*{remark}{Remark}

\newcommand{\Cech}{\mathcal{C}}
\renewcommand{\d}{\;\textnormal{d}}

\newcommand{\N}{\mathbb{N}}
\newcommand{\E}{\mathbb{E}}
\newcommand{\R}{\mathbb{R}}

\newcommand{\U}{\mathcal{U}}

\newcommand{\Var}[1]{\textnormal{Var}\left({{#1}}\right)}

\newcommand{\Lap}[1]{\mathcal{L}\left\{{#1}\right\}}

\newcommand{\Rho}{\mathrm{P}}

\title{When Do Two Distributions Yield the Same Expected Euler Characteristic Curve in the Thermodynamic Limit?}
\author{Tobias Fleckenstein \thanks{University of Bonn} \and Niklas Hellmer\thanks{University of Warsaw and Polish Academy of Sciences, email: \texttt{nhellmer at impan.pl}}}

\date{\today}

\begin{document}

\maketitle
\begin{abstract}
    Let $F$ be a probability distribution on $\R^d$ which admits a bounded density.
    We investigate the Euler characteristic of the \v{C}ech complex on $n$ points sampled from $F$ i.i.d. as $n\to\infty$ in the thermodynamic limit regime.
    As a main result, we identify a condition for two probability distributions to yield the same expected Euler characteristic under this construction. 
    Namely, this happens if and only if their densities admit the same excess mass transform.
    Building on work of Bobrowski, we establish a connection between the limiting expected Euler characteristic of any such probability distribution $F$ and the one of the uniform distribution on $[0,1]^d$ through an integral transform. 
    Our approach relies on constructive proofs, offering explicit calculations of expected Euler characteristics in lower dimensions as well as reconstruction of a distribution from its limiting Euler characteristic. 
    In the context of topological data analysis, where the Euler characteristic serves as a summary of the shape of data, we address the inverse problem and determine what can be discriminated using this invariant.
    This research sheds light on the relationship between a probability distribution and topological properties of the \v{C}ech complex on its samples in the thermodynamic limit.
\end{abstract}

\input{Introduction}
\input{Background}%
\input{IntegralTrafo}%
\input{Fibre}%
\input{outlook}%

\section*{Acknowledgements}
This work was initiated while NH was visiting Helmholtz Munich. He gratefully acknowledges the hospitality of his host Bastian Rieck as well as financial support from the University of Warsaw via the IDUB program, area POB 3.
NH was supported by the Dioscuri program initiated by the Max Planck Society,  jointly managed with the National Science Centre (Poland), and mutually funded by the Polish Ministry of Science and Higher Education and the German Federal Ministry of Education and Research. 
We thank Paweł Dłotko, Lennart Ronge and Rafał Topolnicki for helpful comments.

\printbibliography

\end{document}

%% file: Introduction.tex
\section{Introduction}
Topological data analysis (TDA) \cite{edelsbrunnerComputationalTopologyIntroduction2010,wasserman_topological_2018} is a relatively young field of research, which aims to leverage tools from algebraic topology to study ``the shape of data".
One of its most prominent constructions is the \textit{\v{C}ech complex} $\Cech_r(X)$ of a finite point cloud $X\subset \R^d$, which has vertices $X$ and simplices $\sigma\subseteq X$ if the intersection $\bigcap_{x\in\sigma}\overline{B}_r(x)$ is non-empty.
Here, $r\geq 0$ is the filtration parameter, meaning that $\Cech_r(X)\subseteq \Cech_s(X)$ whenever $r\leq s$.
In this work, we are interested in the case when $X = X_n = \{x_1,\ldots,x_n\}$ consists of $n$ i.i.d. samples from some probability distribution $F$ on $\R^d$.
As the sample size $n$ goes to infinity, there are three limiting regimes governing the topology of $\Cech_{r_n}(X_n)$, which are distinguished by the behaviour of $\Lambda_n = n \omega_d r_n^d$.
Here, $\omega_d$ is the volume of a unit ball in $\R^d$ and $(r_n)_n$ is a sequence of parameters of the \v{C}ech complex.
In the dense regime, $\Lambda_n\to\infty$, the \v{C}ech complex is connected; if $\Lambda_n$ grows fast enough, it recovers the topology of the support of $F$ with high probability.
However, no other information about the distribution is kept.
In the thermodynamic regime, $\Lambda_n\to\Lambda \in ]0,\infty[$, on the other hand, we cannot recover the support of $F$ but can hope to capture different information about the distribution.
Finally, in the sparse regime, $\Lambda_n\to0$, the \v{C}ech complex is so disconnected that it retains not much information at all.

This raises the question what properties of the distribution are in fact captured by the topology of the \v{C}ech complex in the thermodynamic limit.
To this end, Vishwanath et al. \cite{vishwanath_limits_2022} have recently introduced the concept of ``$\mathcal{F}$-equivalence'', which provides a sufficient condition for probability distributions to have \v{C}ech complexes which are indistinguishable by means of topological invariants in this regime.
The main result of the present article is to show that this condition is indeed also necessary in the setting of expected Euler characteristic curves.
The two preceding statements can be succinctly combined into the following theorem:
\begin{theorem}
    Let $F,G$ be probability distributions on $\R^d$ with densities with respect to the Lebesgue measure $f,g$ which are bounded.
    The following are equivalent:
    \begin{enumerate}[i)]
        \item The excess mass transforms agree $\hat{f}(t)=\hat{g}(t)$ for all $t>0$,
        \item for any $X\sim F, Y\sim G$ we have $f(X) \overset{D}{=} g(Y)$,
        \item in the thermodynamic limit, the expectations of persistent Betti numbers agree: $\E[\beta_k^{s,t}(F)] = \E[\beta_k^{s,t}(G)]$ for all $k\in\N, 0<s<t$,
        \item in the thermodynamic limit, the expected Euler characteristic curves agree: $\overline{\chi}_F(\Lambda) = \overline{\chi}_G(\Lambda)$ for all $\Lambda>0$. 
    \end{enumerate}
\end{theorem}
The implications $i)\Rightarrow ii)\Rightarrow iii)\Rightarrow iv)$ were established by Vishwanath et al. \cite{vishwanath_limits_2022}, condition \textit{i)} is their notion of ``$\mathcal{F}$-equivalence''.
The subject of the present work is to show the perhaps surprising implication $iv)\Rightarrow i)$.
This is Theorem \ref{thm:fibre} below.

Let us briefly collect some related work.
In the context of the advent of TDA, there has been considerable effort to understand random geometric complexes \cite{kahleRandomGeometricComplexes2011, bobrowskiTopologyRandomGeometric2017}, generalizing the theory of random geometric graphs \cite{penroseRandomGeometricGraphs2003}.
The key idea of TDA is to study the changes of topological invariants when varying this parameter, a concept known as \textit{persistence}.
Thus, the numerical invariant of the Euler characteristic becomes a function of one non-negative real parameter; this is the Euler characteristic curve (ECC), the corresponding algebraic invariant is persistent homology.
TDA follows the slogan that ``data has shape'', but data of course also has a density. 
In this article, we address the question what the shape of the data encodes about its density.
While in the context of TDA, the ECC is used as a functional summary of the data, we are interested in the inverse problem:
Given an ECC, what can we know about the probability distribution governing the data?
We establish means to explicitly compute a possible probability density from the limiting expected ECC.

Pioneering the study of the Euler characteristic of random \v{C}ech complexes was Bobrowski's insight to exploit Morse-theoretic ideas \cite{bobrowski_algebraic_2012, bobrowski_distance_2014}.
Functional laws of large numbers for the ECC were recently presented in \cite{thomas_FSLLN_2021} and \cite{thomasFunctionalLimitTheorems2020b}, which also provides a functional central limit theorem.
This was later extended by \cite{Krebs2021OnAT} and applied to goodness of fit testing \cite{dlotko_topology-driven_2023}.
One major motivation for the present article is the question: Against which distributions does the test \cite{dlotko_topology-driven_2023} have power?
A different aspect of ECCs in a statistical context is its links to percolation theory \cite{bobrowski_homological_2020}.

Another topological invariant is given by Betti numbers, which extend the notion of connectivity to higher dimensions.
They are closely related to the Euler characteristic, which is expressed as the alternating sum of Betti numbers.
In the setting of random geometric complexes, Betti numbers were studied initially by
\cite{kahleRandomGeometricComplexes2011}, then limit theorems and a law of large numbers were established by \cite{yogeshwaran_random_2017} and later strengthened by \cite{goel_strong_2019}.
Of course, these results imply statements about the Euler characteristic via taking the alternating sum.
However, there are more tools available for the Euler characteristic than for Betti numbers, allowing for example more explicit expressions for the limit expectation \cite{bobrowskiTopologyRandomGeometric2017}.

%% file: Background.tex
\section{Background}
Let $F$ be a probability distribution on $\R^d$ which admits a density $f\colon\R^d\to\R$ with respect to the Lebesgue measure.
Throughout, we assume it is bounded, i.e. $\|f\|_\infty <\infty$.
\begin{definition}\label{def:ExcessMass}
    We define the \textit{excess mass transform} of a probability density $f\colon \R^d \to [0,\infty[$ as
\begin{equation}\label{eqn:ExcessMass}
    \hat{f}(t) = \int_{\R^d} \mathbbm{1}_{[t,\infty[}(f(x)) \, f(x) \d x.
\end{equation}    
\end{definition}

It is easy to see that the function $1-\hat{f}$ is the distribution function of the random variable $f(X)$ where $X\sim F$.
Note that our definition is slightly different from M\"uller \& Sawitzki \cite{muller_excess_1991} and Polonik \cite{polonik_measuring_1995}.
See Figure~\ref{fig:ExcessMass} for an illustration.
We shall consider the derivative $\hat{f}'$ in a distributional which is defined via integration by parts; in particular, we can make sense of the Laplace transform \cite{beals_advanced_1973}:
\[
    \Lap{\hat{f}'}(\Lambda) = \int_0^\infty \hat{f}'(y) e^{-\Lambda y}\d y = (0-1)+\Lambda \int_0^\infty  \hat{f}(y) e^{-\Lambda y} \d y = -1+\Lambda\Lap{\hat{f}}(\Lambda).
\]

\begin{figure}
    \centering
    \begin{tikzpicture}[domain=0:8,xscale=1,yscale=3]
        \draw[<->] (0,0.5) node[left]{$y$}-- (0,0) -- (8.5,0) node[below] {$x$};
        \draw[domain=0:8,smooth,red] plot (\x, {1/(sqrt(2*pi))*exp(-(\x-4)^2/(2))}) node[above] {$f(x)$};
        \draw[domain={4-sqrt(1/2)}:{4+sqrt(1/2)}, smooth, variable=\x, fill=red!50, draw=none] plot ({\x},{1/(sqrt(2*pi))*exp(-(\x-4)^2/(2))}) -- ({4+sqrt(1/2)},0) -- ({4-sqrt(1/2)},0);
        \draw[blue] plot [smooth] coordinates {(0,{1/sqrt(2*pi)})  ({0.112463},{exp(-1/100)/sqrt(2*pi)})  ({0.24817},{exp(-1/20)/sqrt(2*pi)})  ({0.382925},{exp(-1/8)/sqrt(2*pi)}) ({0.436279},{exp(-1/6)/sqrt(2*pi)})  ({0.5202},{exp(-1/4)/sqrt(2*pi)})  ({0.682869},{exp(-1/2)/sqrt(2*pi)}) (0.842701,{exp(-1)/sqrt(2*pi)})  (0.9545,{exp(-2)/sqrt(2*pi)})  (1,0)};
        \draw[dashed] (0,{exp(-1/4)/sqrt(2*pi)}) -- ({4+sqrt(1/2)},{exp(-1/4)/sqrt(2*pi)});
        \draw[dotted] ({0.5202},0) -- ({0.5202},{exp(-1/4)/sqrt(2*pi)});
        \draw ({0.5202},0.02)--({0.5202},-0.02);
        \draw (1,0.02)--(1,-0.02);
        \node[left] at (0,{exp(-1/4)/sqrt(2*pi)}) {$t$} ;
        \node[below] at (0.5,0) {$\hat{f}(t)$} ;
        \node[below] at (1,-0.03) {{$1$}};
        \node[red] at ({4},0.1) {$\hat{f}(t)$} ;
        \node[blue] at (0.5,0.5) {$\hat{f}(y)$} ;

    \end{tikzpicture}
    \caption{Illustration of a density (whose domain is the horizontal axis) and its excess mass, which is defined on the vertical axis and takes values on the horizontal axis.\label{fig:ExcessMass}}
\end{figure}
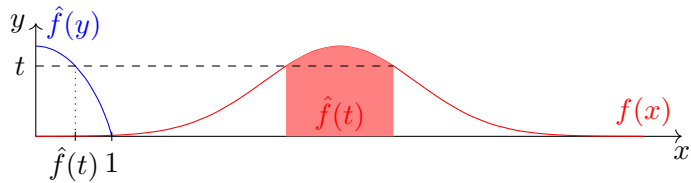

    We are interested in sampling more and more points from $F$, this can be done in the \textit{Bernoulli} or in the \textit{Poisson} setting.
    The former means that we sample $n$ points i.i.d. from $F$.
    The latter means that the sample was generated by a Poisson point process of intensity $nf$.
    In either case, we denote the resulting point cloud by $X_n$.
    Given such a point sample $X_n$, we study the union of closed euclidean balls $\mathcal{O}_{r_n}(X_n) = \bigcup\limits_{x\in X_n} \overline{B}_{r_n}(x)$.
    As we let $n\to \infty$, we consider a sequence of shrinking radii $r_n$ such that $n\omega_d r_n^d \to \Lambda$, where $\omega_d$ is the volume of the $d$-dimensional unit ball.
    Intuitively, $\Lambda$ is the total volume of the collection of balls.
    The case $0<\Lambda<\infty$ is called the \textit{thermodynamic} or \textit{critical regime}; $\Lambda=0$ is the \textit{sparse} and $\Lambda=\infty$ is the \textit{dense} regime.
    We are interested in the topology of the union of balls in the thermodynamic limit regime.
    Specifically, we investigate the \textit{Euler characteristic}.
    This is a topological invariant which can be defined in several ways.
    It is convenient to replace $\mathcal{O}_{r_n}(X_n)$ by an equivalent combinatorial construction, namely the \textit{\v{C}ech complex} $\mathcal{C}_{r_n}(X_n)$.
    This is a geometric simplicial complex, i.e. a collection of vertices, edges, triangles, tetrahedra and so on.
    Thus, it is a generalization of geometric graphs.
    Specifically, the \v{C}ech complex has vertex set $X_n$ and we include a $k$-simplex $\sigma\subseteq X_n$ iff $\bigcap_{x\in \sigma} \overline{B}_{r_n}(x) \neq \emptyset$.
    See Figure \ref{fig:Cech} for an illustration.
    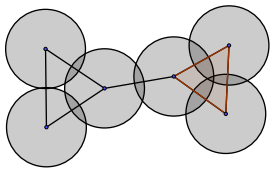
\begin{figure}
    \definecolor{zzttqq}{rgb}{0.6,0.2,0}
	\definecolor{ududff}{rgb}{0.30196078431372547,0.30196078431372547,1}
	\definecolor{bcduew}{rgb}{0.7372549019607844,0.8313725490196079,0.9019607843137255}
	\definecolor{aqaqaq}{rgb}{0.6274509803921569,0.6274509803921569,0.6274509803921569}
        \centering
        {\scalebox{0.25}{\input{Cech.txt}}}
        \caption{The \v{C}ech complex on a sample of six points captures the topology of the union of balls.}
        \label{fig:Cech}
    \end{figure}

    \begin{definition}\label{def:EulerCharacteristic}
        The \textit{Euler characteristic} of the \v{C}ech complex is
        \[
            \chi(\mathcal{C}_{r_n}(X_n)) = \sum\limits_{\sigma\in \mathcal{C}_{r_n}(X_n)} (-1)^{|\sigma|-1},
        \]
        where $|\cdot|$ denotes the cardinality of a set.
    \end{definition}
    The behaviour of the Euler characteristic as $n\to \infty$ was  studied by Bobrowski in his PhD thesis \cite{bobrowski_algebraic_2012} using Morse-theoretic ideas and has a long tradition in stochastic geometry \cite{chiu_stochastic_2013}.
    \begin{definition}\label{def:EECC}
         Let $\overline{\chi}_{F} \colon [0,\infty[ \to \R$ be the function
    \[
        \overline{\chi}_{F}(\Lambda) = \begin{cases}
            1 & \text{if } \Lambda = 0,\\
            \lim\limits_{n\to\infty} n^{-1} \E[\chi(\mathcal{C}_{r_n}(X_n))]&\text{otherwise,}
        \end{cases}        
    \]
    where $n\omega_d r_n\to \Lambda\in]0,\infty[$ as $n\to\infty$.
        We call the function $\overline{\chi}_F$ the \textit{expected Euler characteristic curve}, or \textit{EECC} for short.
    \end{definition}
    The goal of this article is to identify the fibre of the map $F\mapsto \overline{\chi}_F$, which will be done in Theorem \ref{thm:fibre}.
    
    Bobrowski presented a first version of the following result in his thesis \cite{bobrowski_algebraic_2012} and extended it in subsequent work with Mukherjee \cite{bobrowski_topology_2013} to the general setting of manifolds:
    \begin{theorem}[\cite{bobrowski_topology_2013}, Theorem 4.4 and Corollary 4.5]\label{thm:gamma_k}
        Let $f\colon \R^d\to\R$ be a bounded probability density. 
        In the thermodynamic limit,
        \[
            \lim\limits_{n\to\infty} n^{-1} \E[\chi_{n,f}(\Lambda)] = 1 + \sum\limits_{k=1}^{d} (-1)^k \gamma^f_k(\Lambda),
        \]
        where
        \begin{equation}\label{eqn:gamma_k}
            \gamma^f_k(\Lambda) = \frac{\Lambda^k}{\omega_d^{k}(k+1)!} \int_{\R^d}\int_{(\R^d)^k} f^{k+1}(x) h_1^c(0,y)e^{-\Lambda R^d(0,y)f(x)} \d y \d x.
        \end{equation}
    \end{theorem}
    We shall not need the definitions of $h_1^c$ and $R(0,y)$, which can be found in \cite{bobrowski_algebraic_2012}.
Bobrowski and Mukherjee provide explicit formulas for $\gamma_k$ for uniform distributions in dimension up to 3.
In general, the EECC of a uniform distribution is of the form $\overline{\chi}_{\mathcal{U}^d} = e^{-\Lambda} P(\Lambda)$, for a certain polynomial $P(\Lambda) = \sum_{i=0}^d p_i \Lambda^i$ with $p_0=1$ \cite[Corollary 6.2]{bobrowski_vanishing_2017}.
For $d=1,2,3$, they are known explicitly \cite{meckeEulerCharacteristicRelated1991}:
\begin{align*}
    \overline{\chi}_{\mathcal{U}^1}(\Lambda) & = e^{-\Lambda}                                                        \\
    \overline{\chi}_{\mathcal{U}^2}(\Lambda) & = e^{-\Lambda}(1-\Lambda)                                             \\
    \overline{\chi}_{\mathcal{U}^3}(\Lambda) & = e^{-\Lambda} \left(1-3\Lambda + \frac{3\pi^2}{32} \Lambda^2\right).
\end{align*}

If one replaces Euclidean by a more general $p$-distance, analogous results to Theorem~\ref{thm:gamma_k} were established in \cite[Theorem 4.3.1]{thomas_stochastic_2021}. Formulas of the limit expectation for the uniform distribution are provided only for $p=\infty$ in terms of Touchard polynomials \cite[Corollary 4.3.3]{thomas_stochastic_2021}.

%% file: Cech.txt
\begin{tikzpicture}[line cap=round,line join=round,>=triangle 45,x=1cm,y=1cm]
\clip(-12.616640864077574,-9.17579630531208) rectangle (12.698871070317466,8.967939674756556);
\draw [line width=2pt,fill=black,fill opacity=0.2] (-4.3455,1.80544) circle (2.1cm);
\draw [line width=2pt,fill=black,fill opacity=0.2] (-4.2993772823371055,-2.3460119162021873) circle (2.1cm);
\draw [line width=2pt,fill=black,fill opacity=0.2] (-1.2396040745723542,-0.28566211800883773) circle (2.1cm);
\draw [line width=2pt,fill=black,fill opacity=0.2] (2.404447434172601,0.34474341725927676) circle (2.1cm);
\draw [line width=2pt,fill=black,fill opacity=0.2] (5.310463194310983,1.989948106861429) circle (2.1cm);
\draw [line width=2pt,fill=black,fill opacity=0.2] (5.141330001921974,-1.6387276571208862) circle (2.1cm);
\fill[line width=2pt,color=zzttqq,fill=zzttqq,fill opacity=0.10000000149011612] (5.310463194310983,1.989948106861429) -- (5.141330001921974,-1.6387276571208862) -- (2.404447434172601,0.34474341725927676) -- cycle;
\draw [line width=2pt] (-4.3455,1.80544)-- (-4.2993772823371055,-2.3460119162021873);
\draw [line width=2pt] (-4.2993772823371055,-2.3460119162021873)-- (-1.2396040745723542,-0.28566211800883773);
\draw [line width=2pt] (-4.3455,1.80544)-- (-1.2396040745723542,-0.28566211800883773);
\draw [line width=2pt] (-1.2396040745723542,-0.28566211800883773)-- (2.404447434172601,0.34474341725927676);
\draw [line width=2pt] (2.404447434172601,0.34474341725927676)-- (5.310463194310983,1.989948106861429);
\draw [line width=2pt] (5.310463194310983,1.989948106861429)-- (5.141330001921974,-1.6387276571208862);
\draw [line width=2pt] (5.141330001921974,-1.6387276571208862)-- (2.404447434172601,0.34474341725927676);
\draw [line width=2pt,color=zzttqq] (5.310463194310983,1.989948106861429)-- (5.141330001921974,-1.6387276571208862);
\draw [line width=2pt,color=zzttqq] (5.141330001921974,-1.6387276571208862)-- (2.404447434172601,0.34474341725927676);
\draw [line width=2pt,color=zzttqq] (2.404447434172601,0.34474341725927676)-- (5.310463194310983,1.989948106861429);
\begin{scriptsize}
\draw [fill=ududff] (-4.3455,1.80544) circle (2.5pt);
\draw [fill=ududff] (-4.2993772823371055,-2.3460119162021873) circle (2.5pt);
\draw [fill=ududff] (-1.2396040745723542,-0.28566211800883773) circle (2.5pt);
\draw [fill=ududff] (2.404447434172601,0.34474341725927676) circle (2.5pt);
\draw [fill=ududff] (5.310463194310983,1.989948106861429) circle (2.5pt);
\draw [fill=ududff] (5.141330001921974,-1.6387276571208862) circle (2.5pt);
\end{scriptsize}
\end{tikzpicture}

%% file: IntegralTrafo.tex
\section{An Integral Transform Formula}
Throughout, we let $F$ be a probability distribution on $\R^d$ which admits a density $f$ with respect to the Lebesgue measure.
Before we state our theorem, we give some intuitive heuristic motivating it.
Consider a small volume element $A$ around a point $x\in \R^d$.
For a sample of sufficiently large size $n$, the relative amount of points falling into $A$ is roughly $vol(A) f(x)$.
If we choose $A$ small enough, we can replace $f$ by its average value on $A$.
We expect $vol(A) f(x)$ times as many points as from a uniform sample in $A$.
Therefore, also the total volume of the union of balls gets scaled by $f(x)$.
In the thermodynamic limit, we can ignore the effects of points outside $A$.
Then the local contribution of our small region to the EECC $\overline{\chi}_{F}(\Lambda)$ is consequently $f(x) \overline{\chi}_{\mathcal{U}^d}(\Lambda f(x)) vol(A)$.
Letting $A$ become infinitesimally small and integrating over all local contributions now recovers the EECC:

\begin{theorem}\label{thm:density-formulas}
    Let $f\colon \R^d \to \R$ be a bounded probability density.
    Then we have the following formula for the expected ECC in the thermodynamic limit:
    \begin{equation}
        \overline{\chi}_F= \int\limits_{\R^d} f(x) \overline{\chi}_{\U^d}(\Lambda f(x)) \d x .
    \end{equation}
    In addition, we have
    \begin{equation}\label{eqn:EECC-integral-transform}
        \overline{\chi}_F= -\int\limits_{0}^{\|f\|_\infty} \hat{f}'(y) \, \overline{\chi}_{\U^d}(\Lambda y) \d y,
    \end{equation}
    where $\hat{f}'$ is the derivative of the excess mass function, which can be understood in a distributional sense. 
\end{theorem}

\begin{proof}
    We simply rearrange the formula~\ref{eqn:gamma_k} and introduce an integral over $[0,1]^d$ of a constant function, which is just a multiplication by one:
    \begin{align*}
        \gamma_k^f(\Lambda)
         & = \frac{\Lambda^k}{\omega_d^{k}(k+1)!} \int_{\R^d}\int_{(\R^d)^k} f^{k+1}(x) h_1^c(0,y)e^{-\Lambda R^d(0,y)f(x)} \d y \d x                        \\
         & = \int_{\R^d}\frac{\Lambda^k}{\omega_d^{k}(k+1)!} (f(x))^{k+1} \int_{(\R^d)^k} h_1^c(0,y)e^{-\Lambda R^d(0,y)f(x)} \d y \d x                      \\
         & =\int_{\R^d}f(x) \frac{(\Lambda f(x))^k}{\omega_d^{k}(k+1)!} \int_{(\R^d)^k} h_1^c(0,y)e^{- (\Lambda f(x)) R^d(0,y)} \d y \d x                    \\
         & =\int_{\R^d}f(x) \frac{(\Lambda f(x))^k}{\omega_d^{k}(k+1)!} \int_{[0,1]^d}\int_{(\R^d)^k} h_1^c(0,y)e^{- (\Lambda f(x)) R^d(0,y)} \d y \d z \d x \\
         & =\int_{\R^d} f(x) \gamma_k^{\mathcal{U}^d}(\Lambda f(x)) \d x.
    \end{align*}
    The first formula of the theorem then follows by taking an alternating sum as in Theorem~\ref{thm:gamma_k}.

    The second formula follows from the first via the integration by parts.
    Namely, we have
    \begin{align*}
        \int_{\R^d} f(x)\overline{\chi}_{\mathcal{U}^d}(\Lambda f(x)) \d x
         & = \int_{\R^d} f(x)\overline{\chi}_{\mathcal{U}^d}(\Lambda f(x)) \d x - \overline{\chi}_{\mathcal{U}^d}(0) + 1                                                                       \\
         & = 1+\int_{\R^d} f(x) \left[\overline{\chi}_{\mathcal{U}^d}(\Lambda y)\right]_{y=0}^{y=f(x)} \d x                                                                                    \\
         & = 1+\int_{\R^d} f(x) \int_{0}^{f(x)} \Lambda \overline{\chi}'_{\mathcal{U}^d}(\Lambda y) \d y \d x                                                                                  \\
         & = 1+ \int_0^{\|f\|_\infty} \int_{\R^d} f(x)\mathbbm{1}_{f(x)\geq y} \Lambda \overline{\chi}'_{\mathcal{U}^d}(\Lambda y) \d x \d y                                                   \\
         & = 1 + \int_0^{\|f\|_\infty} \Lambda \overline{\chi}'_{\mathcal{U}^d}(\Lambda y) \hat{f}(y) \d y                                                                                     \\
         & = 1 + \left[\hat{f}(y) \overline{\chi}_{\mathcal{U}^d}(\Lambda y)\right]_{y=0}^{y=\|f\|_\infty} - \int_0^{\|f\|_\infty}\hat{f}'(y) \overline{\chi}_{\mathcal{U}^d}(\Lambda y) \d y.
    \end{align*}
    Now, we use $\hat{f}(\|f\|_\infty)=0$ and $\hat{f}(0) \overline{\chi}_{\mathcal{U}^d}(0)=1\cdot 1 =1$ to complete the proof.
\end{proof}

\begin{remark}
    A similar result for Betti numbers was presented in \cite[Theorem 1.1]{goel_strong_2019}.
\end{remark}


\begin{remark}
    If we replace Euclidean balls by more general ones with respect to some $p$-distance, Thomas's thesis \cite[Theorem 4.3.1]{thomas_stochastic_2021} provides an analogous result to Theorem~\ref{thm:gamma_k}, but with an infinite series $\overline{\chi}(t) = \sum _{k=0}^{\infty} (-1)^k \psi_k(t)$,
    where $t$ in the setting of that work relates to ours via $\Lambda = \omega_d t^d$.
    Now from parts (i) and (ii) Lemma 4.2.1 of \cite{thomas_stochastic_2021}, one can infer that $\sum_{k=0}^{\infty} \psi_k(t) \leq \exp ((ct)^d\cdot \omega_d \|f\|_\infty)<\infty$.
    Thus, one can apply Fubini's theorem to obtain Theorem~\ref{thm:density-formulas} in this more general setting as well.
\end{remark}

\begin{remark}
    Our theorem re-establishes that the EECC only depends on the excess mass transform as already found by Vishwanath et al.\cite{vishwanath_limits_2022}.
    They also provide various examples of parametric families of distributions which all have the same excess mass, as well as a theoretical study deriving criteria for such families to have this property.
    Let us only point out an elementary example, namely a constant density $f$ on a compact set $K\subsetneq \R^d$.
    Such a density has excess mass $\hat{f}(y) = \mathbbm{1}_{[1/\lambda^d(K),\infty[}(y)$, where $\lambda^d$ denotes the $d$-dimensional Lebesgue measure.
    Consequently, the EECC of such a density only depends on the measure of its support, but not on its topology.
\end{remark}

\begin{example}
    As a sanity check, we evaluate the integral transform formula for $F=\mathcal{U}^d$.
    Then, $\hat{f}(y) = \mathbbm{1}_{[0,1]}(y)$ and thus $\hat{f}'(y) = \delta(y-1)$.
    Consequently, our formula reads as
    \[
        \overline{\chi}_{\mathcal{U}^d}(\Lambda) = \int_{0}^{1}\delta(y-1) \overline{\chi}_{\mathcal{U}^d}(\Lambda y) \d y = \overline{\chi}_{\mathcal{U}^d}(\Lambda),
    \]
    which is of course tautological.
\end{example}

Expressing the EECC of an arbitrary density as an integral transform of the EECC of a uniform density has important implications for computations and theory.
First, let us state an estimate which is a stability theorem similar to \cite[Theorem 3.1]{Krebs2021OnAT} .
\begin{corollary}\label{cor:stability}
    Let $F, G$ be probability distributions on $\R^d$ admitting densities $f$ and $g$, respectively.
    Then we have  $\|\overline{\chi}_F-\overline{\chi}_G\|_\infty\leq \|\hat{f}'-\hat{g}'\|_1$.
\end{corollary}
\begin{proof} We use that $|\overline{\chi}_{\U^d}(\Lambda y)|\leq 1$ and estimate $\|\overline{\chi}_F-\overline{\chi}_G\|_\infty$ as
    \[
        \left\|\int\limits_{0}^{\infty} (\hat{g}'(y)-\hat{f}'(y)) \, \overline{\chi}_{\U^d}(\Lambda y) \d y \right\|_\infty 
        \leq \sup_\Lambda \int\limits_{0}^{\infty} |\hat{g}'(y)-\hat{f}'(y)| \, |\overline{\chi}_{\U^d}(\Lambda y)| \d y
        \leq \int\limits_{0}^{\infty} |\hat{g}'(y)-\hat{f}'(y)| \d y. \qedhere
    \]
\end{proof}

As a second consequence, we can find formulas for the EECC of probability densities which were previously intractable.
\begin{example}
    Consider the two-dimensional density \[f(x_1,x_2) = \frac{1}{2\pi}\exp\left(-\sqrt{x_1^2+x_2^2}\right).\]
    Due to the rotational symmetry of $f$, an easy application of polar coordinates shows that its excess mass has derivative $\hat{f}'\colon [0,\frac{1}{2\pi}]\to \R,\; \hat{f}'(y) = 2\pi\ln(2\pi y)$.
    Plugging this into our formula yields
    \[
        \overline{\chi}_{F}(\Lambda)
        = -\int\limits_{0}^{\frac{1}{2\pi}} 2\pi\ln(2\pi y)  \, \overline{\chi}_{\U^2}(\Lambda y) \d y
        = -\int\limits_{0}^{\frac{1}{2\pi}} 2\pi\ln(2\pi y)  \exp(-\Lambda y)(1-\Lambda y) \d y
        = \frac{2\pi\left(1-e^{-\frac{\Lambda}{2\pi}}\right)}{\Lambda}.
    \]
\end{example}
\begin{table}
    \centering
    \begin{adjustbox}{max width=1.2\textwidth,center}{
            \footnotesize{
                \begin{tabular}{M{5cm}M{10cm}}
                    f                                                                   & \overline{\chi}_F(\Lambda)                                                                                                                                                                 \\
                    \hline
                    e^{-x}                                                              & \frac{1-e^{-\Lambda}}{\Lambda}                                                                                                                                                             \\
                    \frac{1}{\sqrt{2\pi}}\exp(-x^2/2)                                   & \int\limits_0^{\frac{1}{\sqrt{2\pi}}} \frac{2\exp(-\Lambda y)}{\sqrt{-\log(2\pi y^2)}} \d y                                                                                                \\
                    \frac{1}{2\pi}\exp\left(-\frac{x_1^2+x_2^2}{2}\right)               & \exp\left(-\frac{\Lambda}{2\pi}\right)                                                                                                                                                     \\
                    \frac{1}{2\pi}\left(1+\frac{x_1^2+x_1^2}{n}\right)^{-\frac{n+2}{2}} & -\left(\frac{2\pi}{\Lambda}\right)^{\frac{n}{n+2}} \frac{n}{n+2}\left(\gamma\left(1+\frac{n}{n+2},\frac{\Lambda}{2\pi}\right)-\gamma\left(\frac{n}{n+2},\frac{\Lambda}{2\pi}\right)\right) \\
                    \frac{1}{4\pi}\exp\left(-\frac{(x_1^2+x_2^2+x_3^2)^{3/2}}{3}\right) & \frac{e^{-\Lambda / (4\pi)} (-3 \Lambda^2 \pi - 24\Lambda (-16 + \pi^2) + 32 (-1 + e^{\Lambda / (4\pi)}) \pi (-32 + 3\pi^2))}{128\Lambda}
                \end{tabular}}}
    \end{adjustbox}
    \caption{Probability densities and their expected ECCs.
    Here, $\gamma(a,x) =\int_0^x t^{a-1}e^{-t}\d t$ is the lower incomplete gamma function.
    For plots, see Figure~\ref{fig:EECC}.
    For the one-dimensional standard normal distribution, there is no solution in terms of elementary functions.
    }
    \label{tab:formulae}
\end{table}
\begin{figure}
    \centering
    \includegraphics[width=\textwidth]{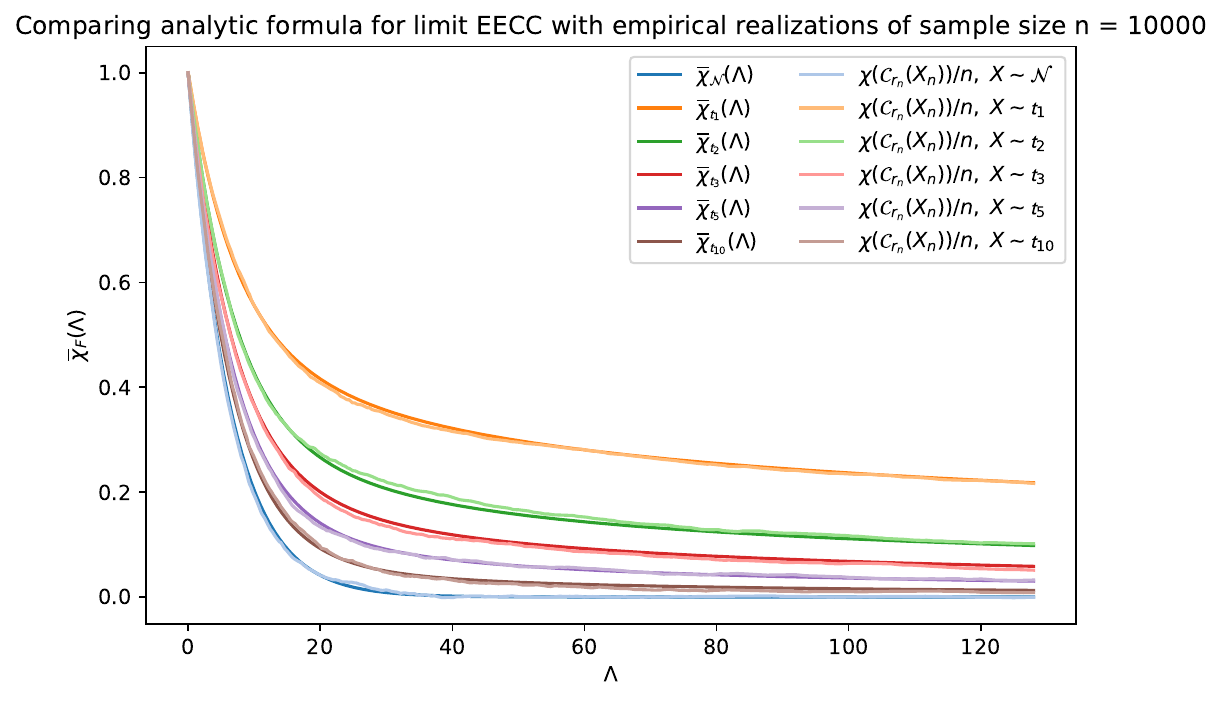}
    \caption{Comparing expected ECCs in the thermodynamic limit of two-dimensional normal and $t$-Student distributions of various degrees of freedom with their empirical counterparts. 
    For the formulas, see Table~\ref{tab:formulae}.
        \label{fig:EECC}}
\end{figure}

See Table~\ref{tab:formulae} for more results and Figure~\ref{fig:EECC} for corresponding plots; we omit the tedious, but straight forward calculus arguments deriving them.
We can observe that the values of the EECCs are strictly positive, which implies that the zeroth Betti number is always greater than the first Betti number.
Intuitively speaking, there are always more connected components than there are holes.
The behaviour is markedly different from the case of the uniform distribution, where the EECC changes signs hinting that different Betti numbers become dominant in different regimes.
The conjecture that this was a general phenomenon~\cite[Section 5.3]{bobrowskiTopologyRandomGeometric2017} is challenged by the result of our computations.
We also show realizations of ECCs from samples of size $n=10000$, which already approximate the limiting EECC quite well.

Note that the EECC of a two-dimensional standard normal distribution coincides with that of a one-dimensional uniform distribution on $[0,1/2\pi]$ (the explicit computation can be found below).
However, the excess masses are different.
If we fix the dimension $d$ this cannot happen, as we shall see next.

%% file: Fibre.tex
\section{Uniqueness of Excess Mass}
In this section we establish a third consequence of Theorem~\ref{thm:density-formulas}, namely that the dependence on the excess mass is injective.
This is to say, for fixed amient dimension $d$, the excess mass is uniquely determined by the expected ECC in the thermodynamic limit.
In fact, we can use Theorem~\ref{thm:density-formulas} to show:
\begin{theorem}\label{thm:fibre}
    Let $F,G$ be probability distributions on $\R^d$ which admit densities
    $f,g\colon \R^d\to\R$ that are bounded.
    Suppose $\overline{\chi}_F(\Lambda) =\overline{\chi}_G(\Lambda)$ for all $\Lambda \geq 0$ and is $d$ times differentiable in $0$.
    Then $\hat{f} = \hat{g}$.
\end{theorem}

Our strategy is to rewrite equation~\ref{eqn:EECC-integral-transform} as an ODE which both Laplace transforms $\Lap{\hat{f}'}$ and $\Lap{\hat{g}'}$ solve.
Indeed, as $\overline{\chi}_{\mathcal{U}^d} = e^{-\Lambda} P(\Lambda)$ for a certain polynomial $P(\Lambda) = \sum_{i=0}^d p_i \Lambda^i$, formula~\ref{eqn:EECC-integral-transform} can be rewritten as
\begin{align*}
    -\overline{\chi}_F(\Lambda) & = \sum_{i=0}^d p_i \Lambda^i \Lap{\hat{f}'(y)y^i}(\Lambda)                        \\
                                & =\sum_{i=0}^d (-1)^i p_i \Lambda^i\frac{d^i}{d\Lambda^i} \Lap{\hat{f}'}(\Lambda),
\end{align*}
using properties of the Laplace transform; see \cite[chapter 7]{beals_advanced_1973} for a textbook introduction.
Then, we will infer that $\hat{f}=\hat{g}$ from the uniqueness of the solution.
In order to carry this idea out, we now derive initial values which only depend on $\overline{\chi}_F=\overline{\chi}_G$ and the ambient dimension.

\begin{lemma}\label{lem:InitialValues}
    \begin{equation}\label{eqn:InitialValues}
        \frac{d^k}{d\Lambda^k} \Lap{\hat{f}'(y)}(0) = (-1)^{k-1}\frac{\overline{\chi}^{(k)}_F(0)}{\sum \limits_{i=0}^k \binom{k}{i} (-1)^i P^{(k-i)}(0)}
    \end{equation}
\end{lemma}

\begin{proof}
    First, we note that the integrand in equation~\ref{eqn:EECC-integral-transform} is continuously differentiable with respect to $\Lambda$,
    whence an application of differentiation under the integral sign yields
    \[
        \overline{\chi}^{(k)}_F(\Lambda) = - \int_\R y^k \hat{f}'(y)e^{-\Lambda y} \sum \limits_{i=0}^k \binom{k}{i} (-1)^i P^{(k-i)}(\Lambda y) \d y.
    \]
    Here, we used the general product formula for
    \[
        \frac{d^k}{d\Lambda^k}(P(\Lambda y) e^{-\Lambda y})
        = \sum_{i=0}^k \binom{k}{i} y^{k-i} P^{(k-i)}(\Lambda y) (-y)^{i}r^{-\Lambda y}
        = y^k e^{-\Lambda y} \sum_{i=0}^k \binom{k}{i} (-1)^{i}P^{(k-i)}(\Lambda y).
    \]
    On the other hand, derivatives of the Laplace transform have the following form:
    \[
        \frac{d^k}{d\Lambda^k}\Lap{\hat{f}'(y)}(\Lambda) = (-1)^k \Lap{y^k \hat{f}'(y)}(\Lambda) = (-1)^k\int_{\R} y^k \hat{f}'(y) e^{-\Lambda y} \d y.
    \]
    Our desired assertion now follows from plugging in $\Lambda = 0$:
    \begin{align*}
        \overline{\chi}^{(k)}_F(0) & = -  \sum \limits_{i=0}^k \binom{k}{i} (-1)^i P^{(k-i)}(0)\int_\R y^k \hat{f}'(y) \d y                        \\
                                   & =(-1)^{k-1}  \sum \limits_{i=0}^k \binom{k}{i} (-1)^i P^{(k-i)}(0)\frac{d^k}{d\Lambda^k}\Lap{\hat{f}'(y)}(0).
    \end{align*}
    Note that we can do this although $\overline{\chi}$ is only defined for $\Lambda\geq 0$ (which means that the derivative is only right-sided) because the right-hand side of equation~\ref{eqn:EECC-integral-transform} is also defined for $\Lambda<0$ and continuously differentiable in $0$.
\end{proof}
\begin{remark}
    It is not hard (employing integration by parts like before) to compute the expression arising in the proof: $\int_\R y^k \hat{f}'(y) \d y = \|f\|_{k+1}^{k+1}$.
    This can be used to derive the bounds $|\frac{d^k}{d\Lambda^k}\Lap{\hat{f}'(y)}(\Lambda)| \leq \|f\|_{k+1}^{k+1}$ and evaluate
    $\frac{d^k}{d\Lambda^k}\Lap{\hat{f}'(y)}(0) = \|f\|_{k+1}^{k+1}$, but we shall not need this result here.
\end{remark}


\begin{proof}[Proof of Theorem~\ref{thm:fibre}]
    Recall that the we can rewrite equation (\ref{eqn:EECC-integral-transform}) from Theorem~\ref{thm:density-formulas} in terms of the Laplace transform as  the following linear ODE:
    \begin{equation}\label{eqn:ODE}
        -\overline{\chi}_F = \sum_{i=0}^d (-1)^i p_i \Lambda^i \frac{d^i}{d\Lambda^i} \Lap{\hat{f}'}.
    \end{equation}
    Here, $P(\Lambda) = \sum_{i=0}^d p_i \Lambda^i$ is the polynomial defined by $\overline{\chi}_{\mathcal{U}^d}(\Lambda) = e^{-\Lambda} P(\Lambda)$.

    Moreover, Lemma~\ref{lem:InitialValues} provides initial values in Equation (\ref{eqn:InitialValues}).
    As $d$ is fixed, so are the coefficients $p_i$ and because $p_0=1$, they are not all zero.
    Therefore, on every compact interval, Picard-Lindel\"of guarantees that $\Lap{\hat{f}'}$ is the unique solution.

    Finally, if $\overline{\chi}_F(\Lambda) =\overline{\chi}_G(\Lambda)$ for all $\Lambda>0$ as in the assumption of Theorem~\ref{thm:fibre},  $\Lap{\hat{f}'}$ and $\Lap{\hat{g}'}$ both satisfy the ODE~\ref{eqn:ODE}.
    In addition, they have the same initial values, given in Equation~\ref{eqn:InitialValues}, which only depend on $\overline{\chi}_F(\Lambda) =\overline{\chi}_G(\Lambda)$ and the ambient dimension.
    Consequently, we infer that $\Lap{\hat{f}'}=\Lap{\hat{g}'}$.
    By injectivity of the Laplace transform, this means $\hat{f}' = \hat{g}'$.
    Now, since $\hat{f}(0)=1=\hat{g}(0)$ because $f$ and $g$ are probability densities, we conclude that $\hat{f} = \hat{g}$, as desired.
\end{proof}

For $d=1,2$, one can write down quite explicit solutions:
In the one-dimensional case, $-\overline{\chi}_F = \Lap{\hat{f}'}$,
so that $\hat{f}(y) = 1-\int_0^y \mathcal{L}^{-1}\{\overline{\chi}_F\}(t) \d t$.
In the two-dimensional case, our differential equation simplifies to
\[
    -\overline{\chi}_F = \frac{d}{d\Lambda}\left(\Lambda \Lap{\hat{f}'}(\Lambda)\right),
\]
and therefore,
\[
    \hat{f} =  \mathcal{L}^{-1} \left\{\frac{1}{s}-\frac{1}{s^2}\int_{0}^{s} \overline{\chi}_F(\Lambda) \d \Lambda \right\}.
\]
While one might like to use these ideas to estimate $\hat{f}$ from empirical estimates of the EECC, this is unfortunately impossible in practice.
The usual Fixed Talbot algorithm \cite{Abate_Laplace_2004} for numerically computing inverse Laplace transforms is numerically quite unstable and cannot handle noisy input data one encounters in empirical EECCs.

However, let us explicitly work out an inverse problem where the EECC is explicitly given.
\begin{example}
    Suppose we are looking for a probability distribution $F$ in $\R^2$ with EECC $\overline{\chi}_F(\Lambda) = e^{-\Lambda}$.
    I.e., the EECC coincides with the one of a uniform distribution on the (one-dimensional) unit interval.
    We obtain
    \begin{align*}
        \hat{f}(y) &= - \mathcal{L}^{-1}\left\{\frac{1}{s} - \frac{1-e^{-s}}{s^2}\right\}(y)\\
         & = (1-y) \mathbbm{1}_{[0,1]}(y).        
    \end{align*}
    To identify a representative with a given excess mass, let us look for one which is radially symmetric around 0, i.e. we can write $f(x_1,x_2) = \rho(x_1^2 + x_2^2)$.
    A straight-forward application of polar coordinates shows that
    \[
        \hat{f}(y) = \pi(\Rho(\rho^{-1}(y))-\Rho(0)),
    \]
    where $\Rho$ is an antiderivative of $\rho$.
    Combining this with our previous information, and taking derivatives on the interval $(0,1)$, we get
    \[
        -1 = \hat{f}'(y) = \frac{\pi y}{\rho'(\rho^{-1}(y))} = \pi y (\rho^{-1}(y))'.
    \]
    Separating the variables yields $\rho(y) = e^{-\pi (y-C)}$.
    The constant $C$ needs to be such that $f$ becomes a probability density, which here means $C=0$, leading us to the solution of the inverse problem as
    \[
        f(x_1,x_2) = e^{-\pi (x_1^2+x_2^2)}.
    \]
\end{example}

\begin{remark}
    If one replaces the Euclidean metric by the supremum distance for the collection of balls, \cite[Corollary 4.3.3]{thomas_stochastic_2021} presents the following expression for the EECC of the uniform distribution \cite[eqn. (4.11)]{thomas_stochastic_2021}:
    \[
        \overline{\chi}_{\mathcal{U}^d}(\Lambda) = -\frac{e^{-\Lambda/\omega_d}}{\Lambda/\omega^d}T_d(-\Lambda/\omega_d).
    \]
    Here, $T_p$ is the Touchard polynomial of degree $p$.
    Now, using the variable $\lambda = \Lambda/\omega_d$, one can argue with the Laplace transform again to establish an analogue to Theorem~\ref{thm:fibre}.
\end{remark}

%% file: outlook.tex
\section{Outlook}
To conclude this paper, we outline two major directions for future research.

First, having established a necessary condition for the expected ECCs to coincide raises the question whether this condition is also necessary in order for the centered ECCs to coincide in distribution (Vishwanath et al. \cite{vishwanath_limits_2022} showed it to be sufficient).
To this end, it is tempting to try a similar approach for higher moments, starting from variance.
While an analogue of Theorem~\ref{thm:density-formulas} is readily established using the description of $\lim_{n\to\infty} n^{-1}\Var{\chi_F(\Lambda)}$ of \cite{bobrowski_algebraic_2012}, the strategy to prove Theorem~\ref{thm:fibre} cannot be replicated.
This is because, unfortunately, there is no analogous expression to $\overline{\chi}_{\mathcal{U}^d} = e^{-\Lambda} P(\Lambda)$, for a certain polynomial $P(\Lambda) = \sum_{i=0}^d p_i \Lambda^i$.

Second, it would be interesting to have a quantitative version of Theorem~\ref{thm:fibre} in the following sense:
Is it possible to compute (or at least bound) the supremum distance $\|\hat{f}-\hat{g}\|_\infty$ in terms of expected ECCs?
Recall that $1-\hat{f}$ is the cumulative distribution function of the random variable $f(X)$ where $X\sim F$.
Thus, $\|\hat{f}-\hat{g}\|_\infty$ is a Kolmogorov-Smirnov test statistic for the null hypothesis $f(X)\overset{D}{=}g(Y)$, where $X\sim F$, $Y\sim G$.
This could pave the way towards a distribution-free multivariate two sample test using computational topology.
Moreover, such a result would imply that the injective continuous map $\hat{f}' \mapsto \int_0^\infty \hat{f}'(y) \overline{\chi}_{\mathcal{U}}(\Lambda y) \d y$ is in addition a homeomorphism onto its image.